\documentclass{amsart}

\newtheorem{theorem}{Theorem}[section]

\newtheorem{corollary}[theorem]{Corollary}

\newtheorem{lemma}[theorem]{Lemma}

\newtheorem{proposition}[theorem]{Proposition}

\newtheorem{example}[theorem]{Example}

\newtheorem{remark}[theorem]{Remark}

\newtheorem{definition}[theorem]{Definition}
\usepackage{scalerel}
\usepackage{amssymb}
\usepackage{amsfonts}
\usepackage{color}
\usepackage{mathtools}
\usepackage{tikz} 
\usepackage{ytableau}
\usepackage{pgfplots}
\usepackage{hyperref}
\usepackage{xcolor}

\usepackage{mathtools}
\usepackage{comment} 

\DeclarePairedDelimiter\ceil{\lceil}{\rceil}

\newcommand{\Hilb}{\mathrm{Hilb}}

\newcommand{\Dim}{\mathrm{dim}}

\author{Yuze Luan}
\begin{document}

\title{Irreducible components of Hilbert scheme of points on non-reduced curves}

\begin{abstract}
We classify the irreducible components of the Hilbert scheme of $n$ points on non-reduced algebraic plane curves, and give a formula for the multiplicities of the irreducible components. The irreducible components are indexed by partitions of $n$; all have dimension $n$; and their multiplicities are given as a polynomial of the parts of the corresponding partitions.
\end{abstract}

\maketitle

\tableofcontents

\section{Introduction}
We study the irreducible components of the Hilbert scheme of $n$ points on a planar curve that could be reducible and non-reduced. We define the curve as follows. Let $f(x,y) = \prod_{1 \leq j \leq r} \limits f_j^{\beta_j}(x,y)$, $r \in \mathbb{Z}_{\geq 0}$ be a product of $r$ irreducible polynomials $f_j$ with multiplicities $\beta_j$. The curve $C$ is defined by the equation $\{f(x,y)=0\}$, where each $f_j^{\beta_j}(x,y)=0$ defines a component $C_j$ with multiplicity $\beta_j$.

We study the irreducible components of the Hilbert scheme of $n$ points on $C$, denoted by $\Hilb^n(C)$. (See Definition \ref{def33}) 

\bigskip

When $C$ is irreducible and reduced (it has only one component $C_1$ of multiplicity $\beta_1=1$), then $\Hilb^n(C)$ is irreducible. This was proven in \cite{AIK} and \cite{AS}.

When $C$ has several components $C_j$ and all the components are reduced (in other words, all the $\beta_j$ are equal to $1$), the irreducible components of $\Hilb^n(C)$ are classified in \cite[Proposition 2.7]{Kivinen}. See also \cite[Fact 2.4]{MRV}. The irreducible components are given by the closures of the collection of ideals who vanish at $t_j$ distinct points on the curve $C_j$ such that $\sum t_j=n$.

We generalize these results to classifying irreducible components of Hilbert scheme of points on arbitrary non-reduced and reducible plane curves $C$, and give a formula for the multiplicities of the irreducible components. 

\bigskip

Below is the picture that represents a point in $\Hilb^n(C)$. Under the Hilbert-Chow map, any ideal in $\Hilb^n(C)$ maps to a finite collection of points $x_1^1, x_2^1, \dots, x_1^2, x_2^2, \dots $ on $C_1, \dots, C_r$, as drawn in the picture. 

Denote the number of points on $C_j$ as $t_j$. Denote multiplicities of $x_1^j, x_2^j, \dots x_{t_j}^j$ by $m_1^j, m_2^j, \dots, m_{t_j}^j$, then $\sum_{i,j}m_i^j=n$. The superscript $j$ of a point $x_i^j$ denotes which curve it sits on, and the subscript $i$ labels from $1$ to $t_j$ the points on each curve $C_j$.

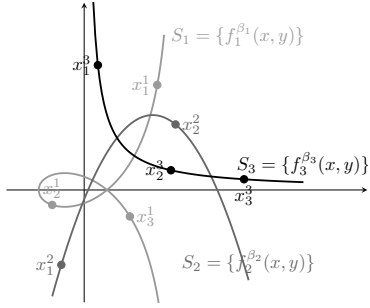
\begin{center}

\begin{tikzpicture}[scale=0.7]
 \begin{axis}[
            xmin=-1.7, xmax=6.2, 
            ymin=-3, ymax=5,
            axis x line = center, 
            axis y line = center,
            xtick = \empty,
            ytick = \empty]
            \addplot[line width=1pt, samples=100,black!40!white, domain=-1:1.75] {(x-0.5)*sqrt((x+1)/(2-x))} node [anchor=west] {$S_1=\{f_1^{\beta_1}(x,y)\}$};
						\addplot[line width=1pt, samples=100,black!40!white, domain=-1:1.9] {-(x-0.5)*sqrt((x+1)/(2-x))};
						\addplot[line width=1pt, samples=100,black!60!white, domain=-0.7:3.6] {-(x-1.5)^2+2}node [anchor=south] {$S_2=\{f_2^{\beta_2}(x,y)\}$};
						\addplot[line width=1pt, samples=100, domain=0.1:4.8] {1/x}node [anchor=south] {$S_3=\{f_3^{\beta_3}(x,y)\}$};

\addplot [only marks,samples at={1.6},black!40!white] {(x-0.5)*sqrt((x+1)/(2-x))} node[anchor=east] {$x_1^1$};
\addplot [only marks,samples at={1},black!40!white] {-(x-0.5)*sqrt((x+1)/(2-x))} node[anchor=west] {$x_3^1$};
\addplot [only marks,samples at={-0.7},black!40!white] {(x-0.5)*sqrt((x+1)/(2-x))} node[anchor=south] {$x_2^1$};

\addplot [only marks,samples at={-0.5},black!60!white] {-(x-1.5)^2+2} node[anchor=east] {$x_1^2$};
\addplot [only marks,samples at={2},black!60!white] {-(x-1.5)^2+2} node[anchor=west] {$x_2^2$};

\addplot [only marks,samples at={3.5}] {1/x} node[anchor=north] {$x_3^3$};
\addplot [only marks,samples at={1.9}] {1/x} node[anchor=east] {$x_2^3$};
\addplot [only marks,samples at={0.3}] {1/x} node[anchor=east] {$x_1^3$};
\end{axis}
\end{tikzpicture}

\end{center}

\begin{theorem}
\label{thm111}
The irreducible components of $\Hilb^n(C)$ have dimension exactly $n$ and are indexed by partitions $\{m_1 ^1, \dots, m_{t_1}^1, \cdots, m_{1}^r , \dots, m_{t_r}^r\}$ of $n$ where $m_i^j$ is the multiplicity of the point $x_i^j$, and each multiplicity satisfies $m_i^j \leq \beta_j$ for all $i$ and $j$.

The irreducible components are given by the closures of the following strata: consider the Hilbert-Chow map from $\Hilb^n(C)$ to the set of point-configurations on $C$, take the preimage in $\Hilb^n(C)$ of any point-configuration where each fat point $x_i^j$ is on the smooth part of $C_j$, and the multiplicities $m_i^j$ satisfy $1\leq m_i^j \leq \beta_j$ for all $i$ and $j$.

\end{theorem}

\begin{theorem}
\label{thm222}
The multiplicity of the stratum indexed by the partition 

$\{m_1 ^1, \dots, m_{t_1}^1, \cdots, m_{1}^r , \dots, m_{t_r}^r\}$ is equal to $\Pi_{i,j} (\beta_j- m_i^j +1)$.
\end{theorem}

We outline the proofs as follows. For the proofs of both theorems, we first solve the easier problem of studying the Hilbert scheme of points on the thick line $\{ y^\beta =0\}$ where $\beta$ is some positive integer, and then extend our results to $\Hilb^n(C)$ for general plane curves $C$.

We stratify $\Hilb^n(C)$ using a natural stratification of $\Hilb^n(\mathbb{C}^2)$. The space $\Hilb^n(C)$ can be embedded into $\Hilb^n(\mathbb{C}^2)$, and the Hilbert-Chow map on $\Hilb^n(C)$ is the composition of the embedding with the Hilbert-Chow map from $\Hilb(\mathbb{C}^2)$ to $S^n\mathbb{C}^2$. We stratify $\Hilb^n(C)$ by the multiplicities of the points in $S^n\mathbb{C}^2$ under the Hilbert-Chow map.

\subsection{Sketch of proof of Theorem \ref{thm111}}

In the proof of Theorem \ref{thm111}, we first show in Lemma \ref{lem_dim_geq_n} that all the irreducible components of $\Hilb^n(C)$ have dimension at least $n$. Additionally, our stratification of $\Hilb^n(C)$ has the property that all the strata are irreducible and have dimension less than or equal to $n$. So the irreducible components must correspond to strata of dimension exactly $n$.

Another tool we use is the affine charts (indexed by partitions) covering $\Hilb(\mathbb{C}^2)$\cite{Haiman}, which give us local coordinate systems on $\Hilb^n(\{ y^\beta =0\})$. We make the observation in Corollary \ref{corU_n} that all the irreducible components must intersect one special chart $U_{(n)}$ corresponding to the partition $(n)$. Then we use the coordinates to find the dimensions of the intersections of the chart $U_{(n)}$ with all the strata, and find that the dimensions of the strata must be $n$ or less (this is discussed in detail in the proof of Theorem \ref{thmZ}). The dimension of a stratum is $n$ if and only if the multiplicities indexing the stratum are all less than or equal to $\beta$. This completes the proof of Theorem \ref{thmZ}, which states that a stratum' closure is an irreducible component of $\Hilb^n(\{ y^\beta =0\})$ precisely when the multiplicities of the points are less than or equal to $\beta$. 

We generalize the above discussion of $\Hilb^n(\{ y^\beta =0\})$ to $\Hilb^n(C)$ by arguing that a collection of points on the smooth part of $C$ is locally the same as a collection of points on $\{ y^\beta =0\}$ with the same multiplicities, and therefore the same condition on the multiplicities applies for the ideals whose images under the Hilbert-Chow map land on the smooth part of the curve $C$. When the images of some ideals land on the singularities, we prove that the dimension of the stratum of these ideals is strictly less than $n$, so it cannot be an irreducible component. Therefore, we conclude that the irreducible components of $\Hilb^n(C)$ are closures of the strata whose images are a collection of points on the smooth part of $C$ and the multiplicities of the points are less than or equal to the corresponding multiplicities of the curves they land on.

\subsection{Sketch of proof of Theorem \ref{thm222}}

In the proof of Theorem \ref{thm222}, we first compute the multiplicity of the component $\overline{\Sigma_{(n)}}$ of $\Hilb^n(\{ y^\beta =0\})$ by computing the dimension of the coordinate ring of $\Sigma_{(n)}$ in the chart $U_{(n)}$ localized at a generic point. Then we show that coordinate ring of any stratum $\Sigma_{\mu_1, \dots, \mu_r}$ of $\Hilb^n(\{ y^\beta =0\})$ can be written as a tensor product of the coordinate rings of $\Sigma_{(\mu_i)}$, and therefore the multiplicity of $\Sigma_{\mu_1, \dots, \mu_r}$ is a product of the multiplicities of the strata $\Sigma_{(\mu_i)}$. This computation also extends to the case of any general component of $\Hilb^n(C)$ where $C$ is any plane curve.

\section{Acknowledgments}

I thank Professor Eugene Gorsky for introducing me to the Hilbert scheme of points and this problem, and being extremely supportive throughout the process of producing this paper. This work was partially supported by the NSF grant DMS-2302305

\section{Background}
\label{stra}
In this section, we define the space we study, some well known maps and coordinate charts on it, and the stratification we use. 

We study the Hilbert scheme of points of the following three spaces: $\mathbb{C}^2$, the thick curve $\{y^\beta =0\}$, and a reducible non-reduced algebraic curve $C$.

\begin{definition} The Hilbert scheme of $n$ points on the plane $\mathbb{C}^2$, denoted as ${\Hilb}^n(\mathbb{C}^2)$, is the set of ideals $I \subset \mathbb{C}[x,y]$ such that $\mathbb{C}[x,y]/I$ has dimension $n$.
\end{definition}

\begin{definition} The Hilbert scheme of $n$ points on the non-reduced curve $\{y^\beta =0\}$, denoted as $\Hilb^n(\{y^\beta =0\})$, is the set of ideals $I \subset \mathbb{C}[x,y]$ such that $\mathbb{C}[x,y]/I$ has dimension $n$ and $I$ contains the polynomial $y^\beta$.
\end{definition}

\begin{definition}
\label{def33}
Let $C\coloneqq \{f=0\}$ be a reducible and non-reduced algebraic curve where $f = \prod_{j}f_j^{\beta_j}(x,y)$ and each $f_j$ is irreducible. The Hilbert scheme of $n$ points on $C$, denoted by $\Hilb^n(C)$, is the set of all ideals $I \in \mathbb{C}[x,y]$ such that $\mathbb{C}[x,y]/I$ has dimension $n$ and $f \in I$. We define $C^{red}$ as the underlying reduced curve of $C$, i.e. $C^{red}= \{ \prod_{j}f_j(x,y)=0 \}$. 
\end{definition}

\bigskip

\bigskip

We define some well-known constructions on the Hilbert scheme of points. 

\begin{definition}
The Hilbert-Chow map is the map $\pi:\Hilb^n(\mathbb{C}^2) \rightarrow S^n\mathbb{C}^n$ that sends an ideal $I$ to its zero locus $V(I)$ with multiplicities. In particular, if $I$ vanishes at points $x_1, \dots, x_k \in \mathbb{C}^2$ with multiplicities $d_1, \dots, d_k$, then we denote the image under the Hilbert-Chow map by $\pi(I)=d_1\cdot x_1 + \dots + d_k\cdot x_k$.
\end{definition}

\begin{definition}
We define the Hilbert-Chow map $\pi: \Hilb^n(C) \rightarrow S^n\mathbb{C}^n$ by composing the embedding $ \Hilb^n(C) \rightarrow \Hilb^n(\mathbb{C}^2)$ with $\Hilb^n(\mathbb{C}^2) \rightarrow S^n\mathbb{C}^n$. 
\end{definition}

\begin{remark}
Under the Hilbert-Chow map $\pi$, the image of any ideal $I$ in $\Hilb^n(C)$ is a collection of points $\{x_1, \dots, x_k\}$ on $C^{red}$, with the sum of multiplicities $d_1 + \dots + d_k = n$.
\end{remark}

\begin{definition}
The punctual Hilbert scheme is $\Hilb^k(\mathbb{C}^2, 0) \coloneqq \pi^{-1}(k \cdot (0,0))$, the preimage under $\pi$ of the point $(0,0)$ with multiplicity $k$ .
\end{definition}

\begin{theorem}
\label{thmdimphilb}
\cite{Briancon} The punctual Hilbert scheme $\Hilb^k(\mathbb{C}^2, 0)$ is irreducible and has dimension $k-1$.
\end{theorem}

\bigskip

\bigskip

We now introduce the stratification we use throughout the paper. 

\begin{definition}
\label{defstraa}
We stratify $\Hilb^n(C)$ by the multiplicities of points in the image of the Hilbert-Chow map $\pi$. Let $m_1^1, \dots,m_i^j, \dots, m_{t_j}^j$ be a partition of $n$. Define each stratum $\Sigma_{m_1^1, \dots, m_{t_j}^j}$ as 
the preimage under the Hilbert-Chow map of $\{m_1^1x_1^1 + \dots m_{t_j}^j x_{t_j}^j\}$, the collections of all possible configuration of $s$ points on the smooth part of the curve $C^{red}$ with multiplicities $m_1^1, \ldots, m_{t_j}^j$. 

Define the stratum $M$ as the set of ideals $I \in \Hilb^n(C)$ such that $\pi(I)$ contains singularities of $C$.
\end{definition}

\begin{remark}
In the study of the easier problem of the irreducible components of $\Hilb^n(\{y^\beta =0\})$, we simplify our notation and denote the multiplicity of points on the reduced smooth line $\{y=0\}$ as $m_1, \dots, m_s$ . The strata are therefore denoted as $\Sigma_{m_1, \dots, m_s}$. The stratum $M$ is empty.
\end{remark}

The strata defined above are not to be confused with the following Definition \ref{defstra}, which will also be useful in the later section. The distinction is that for any ideal in $\Sigma_{m_1, \dots, m_s}$, we allow the location of the collection of points to vary as long as the multiplicities of the points are unchanged, but all the images of ideals in $\Sigma_{m_1, \dots, m_s}(x_1, \dots, x_s)$ must have exactly the same locations and multiplicities $m_1(x_1,0) + \dots m_s (x_s,0)$ under the Hilbert-Chow map. We also have the following Lemma \ref{lemmahcirreducible} computing the dimension of $\Sigma_{m_1, \dots, m_s}(x_1, \dots, x_s)$.

\begin{definition}
\label{defstra}
Fix $s$ distinct points $(x_1, 0), \dots , (x_s,0)$ on the line $\{y^\beta=0\}$ with multiplicities $m_1, \dots, m_s$. We denote their preimage under the Hilbert-Chow map to be $\Sigma_{m_1, \dots, m_s}(x_1, \dots, x_s)$.
\end{definition}

\begin{lemma}

\label{lemmahcirreducible}
The set $\Sigma_{m_1, \dots, m_s}(x_1, \dots, x_s)$ is an irreducible variety, isomorphic to $\Hilb^{m_1}(\mathbb{C}^2, 0) \times \dots \times \Hilb^{m_s}(\mathbb{C}^2, 0)$, and has dimension $n-s$.

\end{lemma}

\begin{proof}

The preimage of each point $m_i(x_i,0)$ under the Hilbert-Chow map is isomorphic to $\Hilb^{m_i}(\mathbb{C}^2, 0)$. The ideals that vanish at all of the points $m_1(x_1,0)+ \dots + m_s(x_s,0)$ must be in one-to-one correspondence with ideals in the product $\Hilb^{m_1}(\mathbb{C}^2, 0) \times \dots \times \Hilb^{m_s}(\mathbb{C}^2, 0)$. And each punctual Hilbert scheme $\Hilb^{m_i}(\mathbb{C}^2, 0)$ is irreducible and has dimension $m_i-1$ (Theorem \ref{thmdimphilb}), so their product is also irreducible, and has dimension $(m_1-1)+(m_2-1)+\dots+(m_s-1) = m_1+m_2+\dots+m_s-s=n-s$.

\end{proof}

\bigskip

\bigskip

\section{Easier problem: studying the components of $\Hilb^n(\{y^\beta =0\})$.}

\bigskip

\bigskip

\subsection{Charts and the lower bound}

\
\\

We first define and state a theorem about the affine charts of $\Hilb^n(\mathbb{C}^2)$ following Haiman's paper \cite{Haiman}.

\begin{definition}
Let $\mu$ be a partition of $n$. Fill the corresponding Young diagram of $\mu$ by monomials, where the $(i,j)$th box in the Young diagram is filled with the monomial $x^iy^j$. We denote collection of monomials by $B_\mu$. 

Define $U_\mu$ to be the set of all ideals $I \in \Hilb^n(\mathbb{C}^2)$ such that $B_\mu$ is a basis of $\mathbb{C}[x,y]/I$. 
\end{definition}

\begin{theorem}\cite{Haiman} \label{thm1} The collection of all $U_{\mu}$, where $\mu$ is a partition of $n$, forms an open cover of $\Hilb^n(\mathbb{C}^2)$. Each chart $U_\mu$ is open, irreducible, smooth, and affine of dimension $2n$.
\end{theorem} 

\bigskip
\bigskip

Now, we put a lower bound on the dimension of the irreducible components of $\Hilb^n(C)$.

\begin{lemma}
\label{lem_dim_geq_n}
The irreducible components of $\Hilb^n(C)$ have dimension at least $n$. 
\end{lemma}

\begin{proof}
Fix a chart $U_\mu$ which has non-empty intersection with $\Hilb^n(C)$. Let $f$ denote the defining polynomial of curve $C$. Because $f \in I$, we can write $f$ as a linear combination of the monomial basis $B_\mu$ mod $I$, and the coefficients in this linear combination should all be $0$. There are $n$ basis elements in $B_\mu$, so there are $n$ conditions imposed on the $2n$ coordinates of $U_\mu$, making the dimension of the irreducible components of $\Hilb^n(C) \cap U_\mu$ at least $n$. And the irreducible components of $\Hilb^n(C)$ should have the same dimension as the irreducible components of $\Hilb^n(C) \cap U_\mu$, because intersecting an irreducible component with an open set $U_\mu$ does not change its dimension. So each irreducible component of $\Hilb^n(C)$ has dimension at least $n$.
\end{proof}

\bigskip

\bigskip

\subsection{The special affine chart where everything happens: $U_{(n)}$} 

\
\\

From now on through the end of this section, we focus on studying the easier problem, the irreducible components of $\Hilb^n(\{y^\beta =0\})$.

\bigskip
\bigskip

The only chart relevant to our computation is $U_{(n)}$, the chart indexed by the partition $(n)$. In the notation of the Young diagram, this partition corresponds to a row of $n$ boxes, and the corresponding $B_{(n)} = \{1, x, x^2, \dots, x^{n-1}\}$.

\begin{center}
\begin{tikzpicture}[scale=0.75]
\draw (0,0)--(0,1)--(5,1)--(5,0)--(0,0);
\draw (1,0)--(1,1);
\draw (2,0)--(2,1);
\draw (3,0)--(3,1);
\draw (4,0)--(4,1);
\draw (0.5,0.45) node {$1$};
\draw (1.5,0.45) node {$x$};
\draw (2.5,0.5) node {$x^2$};
\draw (3.5,0.5) node {$\dots $};
\draw (4.5,0.5) node {$\footnotesize x^{n-1}$};
\end{tikzpicture}
\end{center}

Now, we reveal to the reader why we only need this one specific chart $U_{(n)}$. It is given as the corollary of the following Lemma. 

\begin{lemma}
(a)\label{cordimless}
The space $\Sigma_{m_1, \dots, m_s}(x_1, \dots, x_s) - U_{(n)}$ has dimension strictly less than $n-s$.

(b)The complement of $U_{(n)}$ in each stratum $\Sigma_{m_1, \dots, m_s}$ has dimension strictly less than $n$. Equivalently, $dim(\Sigma_{m_1, \dots, m_s} - U_{(n)}) < n$. 
\end{lemma}

\begin{proof}
(a)
The chart $U_{(n)}$ is open in each $\Sigma_{m_1, \dots, m_s}(x_1, \dots, x_s)$. The intersection $U_{(n)} \cap \Sigma_{m_1, \dots, m_s}(x_1, \dots, x_s)$ is nonempty because the ideal $I = ((x-x_1)^{m_1}\cdot \dots \cdot (x-x_s)^{m_s}, y)$ is in the intersection. Therefore, $\Sigma_{m_1, \dots, m_s}(x_1, \dots, x_s) -U_{(n)}$  is a closed and proper subset of $\Sigma_{m_1, \dots, m_s}(x_1, \dots, x_s)$, and $\Sigma_{m_1, \dots, m_s}(x_1, \dots, x_s)$ is irreducible by Lemma \ref{lemmahcirreducible}. So $\Sigma_{m_1, \dots, m_s}(x_1, \dots, x_s) -U_{(n)}$ has dimension strictly less than $n-s$. 

(b)
Varying each $x_i$ adds $1$ degree of freedom, and varying all the $x_1, \dots, x_s$ adds $s$ degrees of freedom in total.
By part (a), $\Sigma_{m_1, \dots, m_s}(x_1, \dots, x_s) - U_{(n)}$ has dimension strictly less than $n-s$, and therefore $\Sigma_{m_1, \dots, m_s} - U_{(n)}$ has dimension strictly less than $n-s+s= n$.
\end{proof}

\begin{corollary}
\label{corU_n}
All irreducible components of $\Hilb^n(\{y^\beta =0\})$ intersect the chart $U_{(n)}$.
\end{corollary}

\begin{proof}
We want to show that $\Hilb^n(\{y^\beta =0\}) - U_{(n)}$ does not fully contain any irreducible components of $\Hilb^n(\{y^\beta =0\})$.

Suppose for contradiction that $\Hilb^n(\{y^\beta =0\}) - U_{(n)}$ contains some irreducible component $A$ of $\Hilb^n(\{y^\beta =0\})$, then the dimension of $\Hilb^n(\{y^\beta =0\}) - U_{(n)}$ must be greater than or equal to $n$ by Lemma \ref{lem_dim_geq_n}. 

Now we consider the union of complements of $U_{(n)}$ in every stratum $U:= \cup_{m_1, \dots, m_s}(\Sigma_{m_1, \dots, m_s} - U_{(n)})$. Because the dimension of each $\Sigma_{m_1, \dots, m_s} - U_{(n)}$ is strictly less than $n$ by Lemma \ref{cordimless}, the union $U$ also has dimension strictly less than $n$. But $\Hilb^n(\{y^\beta =0\}) - U_{(n)}$ is contained in $U$, so we have a contradiction.

\end{proof}

\bigskip
\bigskip

Now we describe the coordinate system on $U_{(n)}$. This also follows from the discussion in Haiman's paper \cite{Haiman}.

Write $x^n$ and $y$ as a linear combination of the basis $B_{(n)} = \{1,x,{x^2},{\ldots}, x^{n-1} \}$ of $\mathbb{C}[x,y]/I$, and denote the coefficients as follows: $$x^{n}=a_0+a_1x+...+a_{n-1}x^{n-1} \mod I$$ $$y=b_0+b_1x+...+b_{n-1}x^{n-1} \mod  I.$$  

Define polynomials $a(x)=x^{n}-a_{n-1}x^{n-1}-...-a_1x-a_0$ and $b(x)=b_{n-1}x^{n-1}+...+b_1x+b_0$, then any ideal $I$ in $U_{(n)}$ is generated as $I =(a(x), y-b(x))$. Throughout this paper, we will use $a(x)$ and $b(x)$ to denote the polynomials defined above. As a remark, $a(x)$ has degree exactly $n$, but $b(x)$ can have any degree less than or equal to $ n-1$.

\bigskip
\bigskip

\subsection{Stratification inside $U_{(n)}$ and classifying the ideals in each stratum}
\
\\

The stratification $\Sigma_{m_1, \dots, m_s}$ of $\Hilb^n(\{y^\beta =0\})$ induces a stratification $$C_{m_1, \dots, m_s} :=\Sigma_{m_1, \dots, m_s} \cap U_{(n)}$$ on $\Hilb^n(\{y^\beta =0\}) \cap U_{(n)}$. And we know from Corollary \ref{corU_n} that all irreducible components of $\Hilb^n(\{y^\beta =0\})$ intersect $U_{(n)}$. Later in Lemma \ref{lemmairrr} we prove that the strata $C_{m_1,\dots, m_s}$ are irreducible. And as we show in this section, because of the nice coordinate system on $U_{(n)}$, we are able to write out specifically the ideals in each stratum of $C_{m_1, \dots, m_s}$ as in Proposition \ref{propstrata}.

\bigskip

\bigskip

\begin{proposition}
\label{prop1}
The condition that $y^\beta$ is contained in $I =(a(x), y-b(x)) \in U_{(n)}$ is equivalent to the condition that the polynomial $a(x)$ divides $b^\beta(x)$.

\end{proposition}

To prove one direction of the proposition that $y^\beta \in I$ implies $a(x) | b^\beta(x)$, we first prove the following lemma:

\begin{lemma}
\label{lem:div}

Let $f(x)$ be a polynomial in $I=(a(x), y-b(x))$ which does not depend on the variable $y$. Then $f(x)$ is divisible by $a(x)$.

\end{lemma}

\begin{proof}
Perform polynomial long division of $f(x)$ by $a(x)$ and we get that $f(x)= a(x) \cdot q(x) +r(x)$ for some polynomial $r(x)$ and $q(x)$. Suppose for the purpose of contradiction that $r(x)$ is not $0$, and denote the degree of $r(x)$ by $r$, $r < n$. We can explicitly write out $r(x)= l_0+l_1x+ \dots +l_rx^r$ for some $l_i \in \mathbb{C}$ and $l_r \neq 0$.

Because both $f(x)$ and $a(x)$ are in $I$, we have that $r(x) = f(x)-a(x) \cdot q(x)$ must also be in $I$, so $r(x) = 0$ mod $I$, and therefore $r(x)$ is a nonzero linear combination of $1,x,\dots, x^r$. But we also know that $B_\mu= \{1, x, x^2, \dots, x^{n-1}\}$ is a basis of $\mathbb{C}[x,y]/I$, contradiction. So $r(x)$ must be $0$ and $a(x)$ divides $f(x)$.
\end{proof}

Now, we are ready to prove Proposition \ref{prop1}.

\begin{proof}[Proof of Proposition \ref{prop1}]

Assume $y^\beta \in I$. Because $y-b(x)$ is a generator of $I$, $y=b(x)$ mod $I$, which implies that $b(x)^\beta \in I$ and by Lemma \ref{lem:div}, $b^\beta(x)$ is divisible by $a(x)$.

Suppose $a(x) | b^\beta(x)$, then because $a(x) \in I$, we have $b^\beta(x) \in I$. Again because $y=b(x)$ mod $I$, and $b^\beta(x) \in I$, we must have $y^\beta \in I$.

\end{proof}

\bigskip

\bigskip

We factor $a(x)$ and $b(x)$ into linear factors in terms of its roots, and study the possible multiplicities of roots that $a(x)$ and $b(x)$ can have.

\begin{lemma}
\label{lemmaabalpha}

Let $x_1 \dots x_s$ denote the distinct roots of $a(x)$, and $m_i$ the multiplicity of each root $x_i$, then we can explicitly make the factorization: $$a(x)=(x-x_1)^{m_1}\cdot(x-x_2)^{m_2}\cdot \ldots \cdot(x-x_s)^{m_s}\text{ where }\sum_i m_i = n.$$

Then condition $a(x)|b^\beta(x)$ splits into $2$ cases depending on the multiplicities $m_i$. 

\begin{enumerate}\item{General case:}
If $\ceil*{\frac{m_1}{\beta}} + \dots + \ceil*{\frac{m_s}{\beta}} \leq n-1$,\\
 
then $a(x)|b^\beta(x)$ if and only if
$$b(x)=(x-x_1)^{\ceil*{\frac{m_1}{\beta}}}\cdot (x-x_2)^{\ceil*{\frac{m_2}{\beta}}}\cdot \ldots \cdot (x-x_s)^{\ceil*{\frac{m_s}{\beta}}}\cdot \alpha(x)$$ for some polynomial $\alpha(x)$ of degree at most $t= n-1-\sum \ceil*{\frac{m_i}{\beta}}  \ \ (**)$.\\ 

\item{Special case:}
If $\ceil*{\frac{m_1}{\beta}} + \dots + \ceil*{\frac{m_s}{\beta}} > n-1$,

then $a(x)|b^\beta(x)$ if and only if
$b(x)= 0$.

\end{enumerate}
\end{lemma}

\begin{proof}
Denote the multiplicity of $(x-x_i)$ in $b(x)$ by $q_i$. 

Because $a(x)$ divides $b^\beta(x)$, each factor $(x-x_i)$ in $b^
\beta(x)$ must have multiplicity higher than $m_i$, or $b(x)$ has to be $0$. That is to say, the multiplicity $q_i$ must satisfy $\beta \cdot q_i \geq m_i$. Because $q_i$ are integers, the smallest possible value of $q_i$ is $\ceil*{\frac{m_i}{\beta}}$. So $b(x)$ must have the factor $(x-x_1)^{\ceil*{\frac{m_1}{\beta}}}\cdot (x-x_2)^{\ceil*{\frac{m_2}{\beta}}}\ldots \cdot (x-x_s)^{\ceil*{\frac{m_s}{\beta}}}$, or is equal to $0$. 

Recall that $b(x)$ is constructed to have degree at most $n-1$, so if $\ceil*{\frac{m_1}{\beta}} + \dots + \ceil*{\frac{m_s}{\beta}} \leq n-1$, then $\alpha(x)$ is some polynomial of degree at most $n-1 -(\ceil*{\frac{m_1}{\beta}} + \dots + \ceil*{\frac{m_s}{\beta}} )$. The special case happens when $\ceil*{\frac{m_1}{\beta}} + \dots + \ceil*{\frac{m_s}{\beta}} > n-1$, then $b(x)$ has to be $0$.

\end{proof}

\begin{remark}
\label{lemmaSC}

The special case $\ceil*{\frac{m_1}{\beta}} + \dots + \ceil*{\frac{m_s}{\beta}} > n-1$ happens exactly when either (a) $\beta \geq 2$, all $m_i=1$ and $s=n$, or (b) $\beta=1$ and $m_i$ can be any positive integers.

\end{remark}

\begin{proof}

Recall that we assumed $m_i \geq 1$. 
(a) When $\beta \geq 2$, we have $1 \leq \ceil*{\frac{m_i}{\beta}} \leq \ceil*{m_i} = m_i$. We also have $\sum_{i=1}^{s} m_i=n$ as the total degree of $a(x)$. So $\ceil*{\frac{m_1}{\beta}} + \dots + \ceil*{\frac{m_s}{\beta}} \leq n$, and $\ceil*{\frac{m_1}{\beta}} + \dots + \ceil*{\frac{m_s}{\beta}} = n$ only if $\ceil*{\frac{m_i}{\beta}} =m_i$. Finding the possible values of $m_i$ so that $\ceil*{\frac{m_i}{\beta}} =m_i$ is equivalent to finding $m_i$ such that $\frac{m_i}{\beta} \leq m_i < \frac{m_i}{\beta}+1$. For any $\beta \geq 2$, $\frac{m_i}{\beta} \leq m_i$ is always true, and $m_i < \frac{m_i}{\beta}+1$ is equivalent to $m_i < \frac{\beta}{\beta-1}$.

Observe that $1 < \frac{\beta}{\beta-1} \leq 2$ for all $\beta \geq 2$, so the only possible value that  $m_i$ can take is $1$. 

(b) When $\beta=1$, $\ceil*{\frac{m_1}{\beta}} + \dots + \ceil*{\frac{m_s}{\beta}} = \ceil*{m_1} + \dots + \ceil*{m_s} = m_1 + \dots + m_s=n$. So $\ceil*{\frac{m_1}{\beta}} + \dots + \ceil*{\frac{m_s}{\beta}} \geq n-1$ is always satisfied for arbitrary $m_i$. 
\end{proof}

\bigskip

\bigskip

Now, we conclude our results from above and explicitly write out the ideals in each stratum $C_{m_1, \dots, m_s}$.

\begin{proposition}

\label{propstrata}
Each stratum $C_{m_1, \dots, m_s}$ contains exactly the ideals $I$ of the form $I=(a(x), y-b(x))$, where $a(x)=(x-x_1)^{m_1}\cdot(x-x_2)^{m_2}\ldots_\cdot(x-x_s)^{m_s}$, and $b(x)=(x-x_1)^{\ceil*{\frac{m_1}{\beta}}}\cdot (x-x_2)^{\ceil*{\frac{m_2}{\beta}}}\ldots \cdot (x-x_s)^{\ceil*{\frac{m_s}{\beta}}}\cdot \alpha(x)$ when $\ceil*{\frac{m_1}{\beta}} + \dots + \ceil*{\frac{m_s}{\beta}} \leq n-1$ (\textit{general case}); $b(x)=0$ when $\ceil*{\frac{m_1}{\beta}} + \dots + \ceil*{\frac{m_s}{\beta}} > n-1$ (\textit{special case}).
\end{proposition}

\begin{proof}
The ideals $I$ in each stratum $\Sigma_{m_1,\dots, m_s} \cap U_{(n)}$ have the form $(a(x), y-b(x))$, where $a(x)=(x-x_1)^{m_1}\dots (x-x_s)^{m_s}$.
By Proposition \ref{prop1}, finding the ideals $I$ in the intersection of $\Hilb^n(\{y^\beta =0\})$ with $U_{(n)}$ is equivalent to imposing the condition that $b(x) | a^\beta(x)$ for ideals $I=(a(x),y-b(x)) \in U_{(n)}$. By Lemma \ref{lemmaabalpha}, $b(x) | a^\beta(x)$ is equivalent to the condition that $b(x)=(x-x_1)^{\ceil*{\frac{m_1}{\beta}}}\cdot (x-x_2)^{\ceil*{\frac{m_2}{\beta}}}\cdot \ldots \cdot (x-x_s)^{\ceil*{\frac{m_s}{\beta}}}\cdot \alpha(x)$ or $b(x)=0$, depending on the multiplicities $m_i$.
\end{proof}

\bigskip
\bigskip

\subsection{Counting dimension and finding irreducible components.}
\
\\

For each stratum $C_{m_1, \dots, m_s}$, we compute it's dimension by counting the degrees of freedom given by polynomials $a(x)$, $b(x)$, and $\alpha(x)$. 
\begin{lemma}
\label{lemmadim}

If $\ceil*{\frac{m_1}{\beta}} + \dots + \ceil*{\frac{m_s}{\beta}} \leq n-1$, then $\Dim(C_{m_1, \dots, m_s}) = t+s+1$, where $t$ is the maximum degree that $\alpha(x)$ can have as in $(**)$.

For $\beta \geq 2$, there is exactly one stratum $C_{1, \dots, 1}$ satisfying special condition for Lemma \ref{lemmaabalpha} $(2)$, and this stratum $C_{1, \dots, 1}$ has dimension $n$. For $\beta =1$, all strata $C_{m_1, \dots, m_s}$ have dimension $s$.

\end{lemma}

\begin{proof}

Let's first look at the \textit{general case}:$\ceil*{\frac{m_1}{\beta}} + \dots + \ceil*{\frac{m_s}{\beta}} \leq n-1$. 

Each distinct root $x_i$ of $a(x)$ gives a degree of freedom, so $a(x)$ has $s$ degrees of freedom. Denote the maximum degree of $\alpha(x)$ by $t$, then we can explicitly write out $\alpha(x)$ as $\alpha(x)=\alpha_0+ \ldots +\alpha_t x^t$ for coefficients $\alpha_i \in \mathbb{C}$,
\ and each $\alpha_i$ gives a degree of freedom. So $\alpha(x)$ gives $t+1$ degrees of freedom. Note that $b(x)$ is completely determined by $a(x)$ and $\alpha(x)$ so it does not contribute to any degree of freedom. The dimension of stratum $C_{m_1, \dots, m_s}$ therefore is $t+1+s$. 

\textit{Special case}: $\ceil*{\frac{m_1}{\beta}} + \dots + \ceil*{\frac{m_s}{\beta}} > n-1$. As discussed in Lemma \ref{lemmaSC}, in the case of $\beta \geq 2$, we need to have $s=n$ and all the $m_i = 1$, which gives us the stratum $C_{m_1=1, \dots, m_n=1}$.

Recall from Proposition \ref{propstrata} that in the special case, $b(x)=0$, so only $a(x)$ contributes to degrees of freedom, which are given by the variables $x_1, \ldots, x_n$. So the total degree of freedom is $n$. Therefore the dimension of $C_{1, \dots, 1}$ is $n$.

We can also have $\beta =1$. In this case, no matter which stratum we look at, the ideals $I = (a(x), y-b(x))$ in it must satisfy $b(x)=0$ by Proposition \ref{propstrata}. All the degrees of freedom are given by $a(x)$, so the dimension of $C_{m_1, \dots, m_s}$ is $s$.
\end{proof}

\bigskip
\bigskip

\begin{lemma}
\label{lemmairrr}
All the strata $C_{m_1, \dots, m_s}$ are irreducible.
\end{lemma}

\begin{proof}
By Proposition \ref{propstrata} and Lemma \ref{lemmadim}, a stratum $C_{m_1, \dots, m_s}$ is isomorphic to 

{\small $(\mathbb{C}^{t+s+1} -\{(x_1, \dots, x_s, \alpha_0, \dots, \alpha_t) |x_i = x_j$ for some $1 \leq i, j \leq s\}) / Stab(\{m_1, \dots, m_s\})$}. We remove the set of all ideals where some $x_i =x_j$ because the roots should be distinct by construction. We mod out by the action of the stabilizer of the multiplicities to eliminate the over-counting of swapping $x_i$ and $x_j$ when $m_i =m_j$. An affine space with a closed subvariety removed is irreducible, and the quotient by action of a finite group again keeps the space irreducible.
\end{proof}

Now we can conclude that all the closures $\overline{C_{m_1, \dots, m_s}}$ are irreducible.

\bigskip
\bigskip

We conclude all the previous results and classify all the irreducible components of $\Hilb^n(\{y^\beta =0\})$ as the following theorem.

\begin{theorem}
\label{thmZ}
All the irreducible components of $\Hilb^n(\{y^\beta =0\})$ have dimension $n$ and are closures of the strata $C_{m_1, \dots, m_s}$ where $1 \leq m_i \leq \beta$ for all $i$. Given $m_1, \dots, m_s \leq \beta$, the generic point of this component $\overline{C_{m_1, \dots, m_s}}$ consists of $s$ distinct points on $\{y=0\}$ with multiplicities $m_1, \dots, m_s$.
\end{theorem}

\begin{proof}
We remind the reader that in our previous Lemma \ref{lem_dim_geq_n}, we show that the dimensions of the irreducible components are at least $n$. Now we need to find the strata whose dimensions are $n$ or more, and their closures are candidates of the irreducible components.

We first discuss the special case. 

When $\beta =1$, a stratum $C_{m_1, \dots, m_s}$ has dimension $n$ only when $s=n$, so $m_i$ must all be $1$. Therefore, $\Hilb^n(\{y =0\})$ has only one irreducible component $\overline {C_{1, \dots, 1}}$ of dimension $n$. This recovers the previous results of \cite{AIK} and \cite{AS} that $\Hilb^n(\{y =0\})$ is irreducible of dimension $n$.

When $\beta \geq 2$, the stratum $C_{1, \dots, 1}$ has dimension $n$ and it is not a subset of the closure of any other strata. So its closure is an irreducible component of $\Hilb^n(\{y^\beta =0\})$.

Now we look at the general case and want to find the conditions on $t$ and $s$ such that $\Dim(C_{m_1, \dots, m_s}) = t+1+s \geq n$. Recall that we use $t$ to denote the maximum degree that $\alpha(x)$ can have, and $s$ the number of distinct roots of $a(x)$.

Here are all the equations relating the dimension of a stratum $C_{m_1, \dots, m_s}$, $t, s$ and $n$:

$$\dim(C_{m_1, \dots, m_s}) = t+1+s. \text{ (Lemma \ref{lemmadim})}$$

\medskip
Because $b(x)$ has degree at most $n-1$, we have $$t+\ceil*{\frac{m_1}{\beta}}+\ldots + \ceil*{\frac{m_s}{\beta}}
=n-1.$$

Because $m_i \geq 1$, we must have $\ceil*{\frac{m_i}{\beta}} \geq 1$. So $$\ceil*{\frac{m_1}{\beta}}+\ceil*{\frac{m_2}{\beta}}+\ldots + \ceil*{\frac{m_s}{\beta}} \geq s.$$ 

This implies $$\dim(C_{m_1, \dots, m_s})= 1+s+\Bigl( n-1-\Bigl( \ceil*{\frac{m_1}{\beta}}+\ldots + \ceil*{\frac{m_s}{\beta}}\Bigr) \Bigr)$$

$$=n +s - \Bigl( \ceil*{\frac{m_1}{\beta}}+\ldots + \ceil*{\frac{m_s}{\beta}}\Bigr) \leq n. \ \ \ (***)$$ In particular, this implies that the closure of a stratum $\overline{C_{m_1, \dots, m_s}}$ of dimension $n$ are exactly the irreducible components, because there are no other higher dimensional strata.

The equality of the equation $(***)$ holds when $\ceil*{\frac{m_1}{\beta}}+\ldots + \ceil*{\frac{m_s}{\beta}} = s$, which is equivalent to $\ceil*{\frac{m_i}{\beta}}=1$, and this happens precisely when $1 \leq m_i \leq \beta$ for all $i$.

In conclusion, all the strata have dimension $\leq n$, and the closure of a stratum $C_{m_1, \dots, m_s}$ is an irreducible component if and only if $1 \leq m_i \leq \beta$.

\end{proof}

\bigskip
\bigskip

\section{Generalization: studying components of $\Hilb^n(C)$}

Now, we generalize our results to $\Hilb^n(C)$, where $C$ is any non-reduced plane curve.

We remind the reader that we have defined the stratification in Definition \ref{defstraa}. To briefly restate the definition, the stratum $\Sigma_{m_1^1, \dots, m_i^j, \dots, m_{s_r}^r}$ is the set of all ideals $I$ such that the of multiplicity is exactly $m_i^j$ for a point $x_i^j$ on the smooth part of underlying reduced curve $C^{red}_j$ of $C_j$. 

The stratum $M$ is the collection of all ideals whose images contain some singularities of $C$.

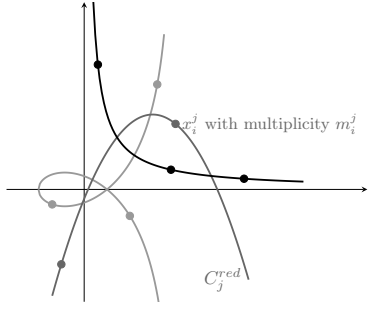
\begin{center}

\begin{tikzpicture}[scale=0.7]
 \begin{axis}[
            xmin=-1.7, xmax=6.2, 
            ymin=-3, ymax=5,
            axis x line = center, 
            axis y line = center,
            xtick = \empty,
            ytick = \empty]
            \addplot[line width=1pt, samples=100,black!40!white, domain=-1:1.75] {(x-0.5)*sqrt((x+1)/(2-x))} node [anchor=west] {};
						\addplot[line width=1pt, samples=100,black!40!white, domain=-1:1.9] {-(x-0.5)*sqrt((x+1)/(2-x))};
						\addplot[line width=1pt, samples=100,black!60!white, domain=-0.7:3.6] {-(x-1.5)^2+2}node [anchor=east] {$C^{red}_j$};
						\addplot[line width=1pt, samples=100, domain=0.1:4.8] {1/x}node [anchor=south] {};

\addplot [only marks,samples at={1.6},black!40!white] {(x-0.5)*sqrt((x+1)/(2-x))} node[anchor=east] {};
\addplot [only marks,samples at={1},black!40!white] {-(x-0.5)*sqrt((x+1)/(2-x))} node[anchor=west] {};
\addplot [only marks,samples at={-0.7},black!40!white] {(x-0.5)*sqrt((x+1)/(2-x))} node[anchor=south] {};

\addplot [only marks,samples at={-0.5},black!60!white] {-(x-1.5)^2+2} node[anchor=east] {};
\addplot [only marks,samples at={2},black!60!white] {-(x-1.5)^2+2} node[anchor=west] {\small $x_i^j$ with multiplicity $m_i^j$};

\addplot [only marks,samples at={3.5}] {1/x} node[anchor=north] {};
\addplot [only marks,samples at={1.9}] {1/x} node[anchor=east] {};
\addplot [only marks,samples at={0.3}] {1/x} node[anchor=east] {};
\end{axis}
\end{tikzpicture}

\end{center}

\bigskip
\bigskip
Similar to the approach to the easier problem, we also want to show that the closure of each stratum $\overline{\Sigma_{m_1^1, \dots, m_i^j, \ldots, m_{s_r}^r}}$ is irreducible. In order to do this, we embed each stratum $\Sigma_{m_1^1, \dots, m_i^j, \ldots, m_{s_r}^r}$ into another irreducible stratum of a bigger space, and show that they have the same closure.
\begin{definition}
Consider the space $\Hilb^n(\mathbb{C}^2, C^{red,sm})$, the Hilbert scheme of points on $\mathbb{C}^2$ supported on the smooth subset of the reduced curve $C^{red,sm}$. We similarly stratify it using the preimage of the Hilbert-Chow map. Denote each stratum by $L_{m_1^1, \dots, m_i^j, \dots, m_{s_r}^r}$. Define $L_{m_1^1, \dots, m_i^j, \dots, m_{s_r}^r} \coloneqq \pi^{-1}( \sum_{i,j}m_i^jx_i^j)$ where each point $x_i^j$ of multiplicity $m_i^j$ is on the open subset $C_j^{red, sm}$ of the smooth points of the reduced curve $C_j^{red}$. 
\end{definition}

It's an easy check that $\Sigma_{m_1^1, \dots, m_i^j, \ldots, m_{s_r}^r} \subset L_{m_1^1, \dots, m_i^j, \ldots, m_{s_r}^r}$ directly from their definitions. However, the condition of being an ideal in the stratum $\Sigma_{m_1^1, \dots, m_i^j, \ldots, m_{s_r}^r}$, namely containing the equation $f(x,y)$ of the curve $C$, is stronger than the condition of being an ideal in the stratum $L_{m_1^1, \dots, m_i^j, \ldots, m_{s_r}^r}$, namely vanishing at points on the curve $C$. So the set containment $\Sigma_{m_1^1, \dots, m_i^j, \ldots, m_{s_r}^r} \subset L_{m_1^1, \dots, m_i^j, \ldots, m_{s_r}^r}$ is proper. 

Now we continue with the irreducibility argument.

\begin{lemma}
\label{sigmairreducible}
Each stratum $L_{m_1^1, \dots, m_i^j, \ldots, m_{s_r}^r}$ is irreducible of dimension $n$.
\end{lemma}

\begin{proof}
Each stratum $L_{m_1^1, \dots, m_i^j, \ldots, m_{s_r}^r}$ is isomorphic to 
\begin{multline*}
\scriptstyle
\Bigl(\prod_{1 \leq i \leq s_j, 1 \leq j \leq r} \limits \Hilb^{m_i^j}(\mathbb{C}^2, 0) \times 
\bigl( \prod_{j=1 \dots r}\limits (C^{sm}_j)^{ s_j} -\{(x_1^1, \dots, x_{s_r}^r)| x_a^j=x_b^j \text{ for some } 1 \leq a,b \leq s_j \} \bigr)\Bigr)\\
/ \scriptstyle \prod_{j=1 \dots r} \limits Stab(m_1^j, \dots, m_{s_j}^j).
\end{multline*}

The preimage of a point with multiplicity $m_i^j$ under the Hilbert-Chow map is isomorphic to $\Hilb^{m_i^j}(\mathbb{C}^2, 0)$. Because the points in the image of $L_{m_1^1, \dots, m_i^j, \ldots, m_{s_r}^r}$ can land anywhere on $C^{red,sm}_j$, as long as they don't collide, we multiply by the factor $\bigl( \prod_{j=1 \dots r}\limits (C^{sm}_j)^{ s_j} -\{(x_1^1, \dots, x_{s_r}^r)| x_a^j=x_b^j \text{ for some } 1 \leq a,b \leq s_j \} \bigr)$. We also need to mod out by the stabilizer of the multiplicities to account for the over-counting when $m_i^j = m_{i'}^j$, and $x_i^j$ and $x_{i'}^j$ are interchanged.

The curve $C^{red}_j$ is irreducible, and its points of singularities form a closed set, so the smooth part of the curve $C^{red,sm}_j$ is irreducible. Removing a closed set from the product of all $C^{red,sm}_j$ leaves the product irreducible. By Theorem \ref{thmdimphilb}, the punctual Hilbert scheme is irreducible, so the product of all these factors is also irreducible. This irreducible product taking the quotient by a finite group is again irreducible.

Now, we want to show that $L_{m_1^1, \dots, m_i^j, \ldots, m_{s_r}^r}$ has dimension $n$. By Theorem \ref{thmdimphilb}, the product of the punctual Hilbert schemes has dimension $m_1^1-1+ \ldots + m_{s_r}^r -1 = m_1^1 + \ldots + m_{s_r}^r -r$. The product of $r$ planar curves with some closed subsets removed gives $r$ more degrees of freedom. Quotienting out by the action of a finite group does not change the dimension. So $L_{m_1^1, \dots, m_i^j, \ldots, m_{s_r}^r}$ has dimension $n-r+r=n$.

\end{proof}

\begin{lemma}
\label{lemma666}
When $1 \leq m_i^j \leq \beta$ for all $i,j$, the closures of the two types of strata are the same: $\overline{\Sigma_{m_1^1, \dots, m_i^j, \ldots, m_{s_r}^r}} = \overline{L_{m_1^1, \dots, m_i^j, \ldots, m_{s_r}^r}}$. Therefore $\overline{\Sigma_{m_1^1, \dots, m_i^j, \ldots, m_{s_r}^r}}$ are irreducible and have dimension $n$. 
\end{lemma}

\begin{proof}
For the proof we use the following fact: 
If $Y$ is a closed subset of an irreducible finite-dimensional topological space $X$, and if $\dim Y= \dim X$, then $Y=X$. Here, we want $Y= \overline{\Sigma_{m_1^1, \dots, m_i^j, \ldots, m_{s_r}^r}}$ and $X=\overline{ L_{m_1^1, \dots, m_i^j, \ldots, m_{s_r}^r}}$, and we want to show that they satisfy the conditions on $X$ and $Y$. 

Assume $m_i^j \leq \beta$. We know from Lemma \ref{sigmairreducible} that the closures of the strata $L_{m_1^1, \dots, m_i^j, \ldots, m_{s_r}^r}$ are closed, irreducible, and have dimension $n$.

A collection of points moving along the smooth part of $C$ are locally the same as the points moving along $\{y^\beta_j = 0\}$, because the local ring at any point on $\{y^\beta_j = 0\}$ is isomorphic to the local ring at any smooth point on $\{f_j^{\beta_j}(x,y)=0 \}$ by locally changing coordinates between $y$ and $f_j(x,y)$. Therefore the dimension of the stratum $\Sigma_{m_1^1, \dots, m_i^j, \ldots, m_{s_r}^r}$ is the sum of the dimension of each stratum $\Sigma_{m_1^j, \ldots, m_{s_j}^j}$ of $\Hilb^n(\{y^\beta =0\})$. When $1 \leq m_i^j \leq \beta_j$, the dimension of each stratum $\Sigma_{m_1^j, \ldots, m_{s_j}^j}$ is $m_1^j+ \ldots+ m_{s_j}^j$ by Theorem \ref{thmZ}. So the dimension of the stratum $\Sigma_{m_1^1, \dots, m_i^j, \ldots, m_{s_r}^r}$ is $m_1^1+ \dots+ m_i^j+ \ldots+ m_{s_r}^r=n$.

So the closures of strata $\overline{\Sigma_{m_1^1, \dots, m_i^j, \ldots, m_{s_r}^r}}$ also have dimension $n$. So the closures of the two types of strata $Y= \overline{\Sigma_{m_1^1, \dots, m_i^j, \ldots, m_{s_r}^r}}$ and $X=\overline{ L_{m_1^1, \dots, m_i^j, \ldots, m_{s_r}^r}}$ satisfy the conditions of being $X$ and $Y$, and therefore they are equal. So $\overline{\Sigma_{m_1^1, \dots, m_i^j, \ldots, m_{s_r}^r}}$ is irreducible of dimension $n$.

\end{proof}

\bigskip
\bigskip

Finally, we have the theorem that classifies the irreducible components of $\Hilb^n(C)$. The reader might notice that we have not discussed if the stratum $M$ is irreducible or not. As it turns out in the proof of the following theorem, $\overline{M}$ is never an irreducible component because its dimension is too small.
\begin{theorem}
The irreducible components of $\Hilb^n(C)$ are the closures of the strata $\overline {\Sigma_{m_1^1, \dots, m_i^j, \ldots, m_{s_r}^r}}$ where $1 \leq m_i^j \leq \beta_j$.
\end{theorem}

\begin{proof}
We first remind the reader of Lemma \ref{lem_dim_geq_n} that we proved in section \ref{stra}: the irreducible components all have dimension $n$ or more.

We look at the two types of strata separately. 
Case 1: The points in the image are all contained in $C^{sm}_j$. The strata are $\Sigma_{m_1^1, \dots, m_i^j, \ldots, m_{s_r}^r}$. 

When $m_i^j > \beta_j$, such strata have dimension strictly less than $n$ by Theorem \ref{thmZ} and a similar argument of locally changing coordinates between $y$ and $f_j(x,y)$ as in the proof of Lemma \ref{lemma666}, and their closures are not the irreducible components. When $1 \leq m_i^j \leq \beta_j$, we have argued in Lemma \ref{lemma666} that such strata are irreducible and have dimension $n$. So their closures must be the irreducible components.

Case 2:  Some of the points in the image $\pi(I)$ are singularities or the intersection points of the curves $C$. We have one stratum $M$ of such ideals. 

The preimage of $s$ points with multiplicities $m_i^j$ on $C_j$ is a subset of the product of the punctual Hilbert scheme $\Hilb^{m_1^1}(\mathbb{C}^2, 0) \times \dots \times \Hilb^{m_{s_r}^r}(\mathbb{C}^2, 0)$. By Theorem \ref{thmdimphilb}, the product of the punctual Hilbert schemes has dimension $m_1^1-1+ \ldots + m_{s_r}^r -1 = m_1^1 + \ldots + m_{s_r}^r -r$. When we allow the points on the smooth part to move, the points at the singularities or the intersections do not move. So the degree of freedom added by moving the points are strictly less than $r$. So the preimage of a collection of points containing some singularity has dimension strictly less than $n$. $M$ is contained in such preimage so has dimension strictly less than $n$ and cannot be irreducible components by Lemma \ref{lem_dim_geq_n}.

\end{proof}

\section{computation of multiplicities of components}

We intersect each stratum $\Sigma_\mu$ with $U_{(n)}$ and the intersection is an open dense subset of each stratum. We study the multiplicities of points in this open dense subset of the intersection.

\bigskip
\bigskip

\subsection{The stratum $\Sigma_{(n)}$ of $\Hilb^n(\{y^\beta =0\})$}
\label{secn}
\
\\

We begin by studying the stratum $\Sigma_{(n)}$ corresponding to the $1$-part partition $(n)$ of $n$. Because we want the closure of this stratum to be an irreducible component, we assume $n \leq \beta$ in this subsection.

We want to study the stratum $\Sigma_{(n)}$ inside the chart $U_{(n)}$, so we pick an ideal $I \in U_{(n)}$, and $I$ is necessarily generated as $I = (a(x), y-b(x))$, where $a(x) = x^n +a_{n-1} x^{n-1} + \dots a_0$ and $b(x) = b_{n-1} x^{n-1} + \dots + b_0$. We also want that $I \in \Sigma_n$, so $a(x) = (x-x_1)^n$ for some $x_1 \in \mathbb{C}$ and $b^\beta(x) =0$ mod $a(x)$. Those are all the conditions we have to consider to compute the coordinate ring of $\Sigma_{(n)}$.




\begin{remark}
We notice that $b^\beta(x) =0$ mod $(x-x_1)^n$ is equivalent to $b^\beta(x+x_1) =0$ mod $x^n$. So we expand the polynomial $b^\beta(x+x_1)$ and set each polynomial coefficient in variables $b_0, \dots, b_{n-1}$ of the term $x^i$ to be $0$ for $0 \leq i \leq n-1$.
\end{remark}

We define the coefficients of $b(x+x_1)$ first before taking its $\beta$-th power. 
 
\begin{definition}

We define the coefficients $B_i$ of $b(x+x_1)=b_0 +b_1(x+x_1)+b_2(x+x_1)^2 + \dots + b_{n-1}(x+x_1)^{n-1}:=  B_0 + B_1x + B_2x^2 + \dots + B_{n-1}x^{n-1}.$

Each $B_i$ is a polynomial of variables $b_i, \dots, b_{n-1}$ and $x_i$: For $0 \leq i \leq n-1$, $$B_i := b_i + {i+1 \choose i} b_{i+1}x_1 + {i+2 \choose i} b_{i+2}x_1^2 + \dots + {n-1 \choose i} b_{n-1}x_1^{n-i-1}.$$ 
\end{definition}

\bigskip

Now we take the $\beta$-th power of $b(x+x_1)$ and find the coefficients of $x^i$ in terms of $B_i$.

\begin{definition}
\label{defEi}
Define $E_i$ as a function of $B_i$'s. Set $$E_i := \sum_{k_0, \dots, k_{n-1}} \limits \binom{\beta}{k_0, \dots, k_{n-1}}B_0^{k_0} \dots B_{n-1}^{k_{n-1}},$$ where $k_0,\dots, k_{n-1}$ satisfy $0 \cdot k_0 + \dots + (n-1) \cdot k_{n-1} = i$ and $k_0 + \dots + k_{n-1} = \beta$.
\end{definition}

\begin{lemma}
The function $b^\beta(x+x_1)$ can be written as  $b^\beta(x+x_1) = \sum_{i =0, \dots, n-1} \limits E_i x^i$.
\end{lemma}

\begin{proof}
By the multinomial theorem,

\begin{flalign*}
b^\beta(x+x_1)
& = (B_0 + B_1x + B_2x^2 + \dots + B_{n-1}x^{n-1})^ \beta \\
& = \sum_{k_0+ \dots+ k_{n-1}=\beta}\limits \binom{\beta}{k_0, \dots, k_{n-1}}B_0^{k_0}( B_1x)^{k_1}\dots (B_{n-1}x^{n-1})^{k_{n-1}}  \\
&= \sum_{k_0+ \dots+ k_{n-1} = \beta}\limits \binom{\beta}{k_0, \dots, k_{n-1}}(B_0^{k_0} B_1^{k_1}\dots B_{n-1}^{k_{n-1}})x^{(0 \cdot k_0 + 1\cdot k_1 + \dots (n-1)\cdot k_{n-1})}. \\
\end{flalign*}

We denote the power of $x$ as $i$, and therefore for each term $x^i$, we have $i = 0 \cdot k_0 + 1\cdot k_1 + \dots (n-1)\cdot k_{n-1}$ and $k_0+ \dots+ k_{n-1} = \beta$. The coefficient of $x^i$ is $E_i$.
\end{proof}

\begin{corollary}
\label{lemieq}

From the computation above, $b^\beta(x+x_1)$ mod $x^n =0$ if and only if $E_i=0$ for $0 \leq i \leq n-1$.

\end{corollary}

\begin{corollary}
Denote the coordinate ring of the component $\Sigma_{(n)}$ as $R_{(n)}$, then $R_{(n)}$ is isomorphic to $$ R_{(n)} := \mathbb{C}[b_0, \dots, b_{n-1}, x_1]/(E_0, \dots, E_{n-1}).$$
\end{corollary}

We also make the following observations about the functions $E_i$.

\begin{lemma}
\label{divisiblelem}

For $0 \leq i \leq n-1$, the function $E_i$ is divisible by $B_0^{\beta -i}$ but not divisible by $B_0^{\beta -i+1}$

\end{lemma}

\begin{proof}
Let's look at a term $\binom{\beta}{k_0, \dots, k_{n-1}}B_0^{k_0} \dots B_{n-1}^{k_{n-1}}$ in $E_i$. In order that $k_0,\dots, k_{n-1}$ satisfy $0 \cdot k_0 + \dots + (n-1) \cdot k_{n-1} = i$ and $k_0 + \dots + k_{n-1} = \beta$ for $i \leq n-1 < \beta$, we must have $k_0 \geq \beta -i$. 

When $i=0$, we have $E_0 = B_0^\beta$ and therefore $E_0$ is not divisible by $B_0^{\beta+1}$.

For every $i$ such that $1 \leq i \leq n-1$, the term $\binom{\beta}{k_0,k_1}B_0^{k_0}B_1^{k_1}$ where $k_0 + k_1= \beta$ and $k_1=i$ is in $E_i$. This term is not divisible by $B_0^{\beta -i+1}$. Each term of $E_i$ is a positive constant times a monomial of $B_0, \dots, B_{n-1}$ and each monomial is different, so if one monomial term is not divisible by $B_0^{\beta -i+1}$, the entire function $E_i$ is not divisible by $B_0^{\beta -i+1}$.
\end{proof}

\begin{corollary}
Given $n \leq \beta$, $E_i$ must be divisible by $B_0$. So $B_0=0$ implies $E_i =0$ for all $i$.
\end{corollary}

\begin{proof}
When $n \leq \beta$, we have that $i \leq n-1 \leq \beta -1$. So $\beta -i \geq 1$, and we must have that $E_i$ is divisible by $B_0$.
\end{proof}

\begin{definition}
Define $\Sigma_{(n)}^{red}$ as the reduced variety corresponding to $\Sigma_{(n)}$.
\end{definition}

\begin{lemma}
The reduced variety $\Sigma_{(n)}^{red}$ is cut out in $U_{(n)}$ by the equations $$B_0= b(x_1)= b_0 +b_1x_1+b_2x_1^2 + \dots + b_{n-1}x_1^{n-1} = 0.$$ and $a(x)= (x-x_1)^n=0.$ In other words, $$\Sigma_{(n)}^{red} :=\{ (b_0, \dots, b_{n-1}, x_1)| b(x_1)= 0,(x-x_1)^n=0 \}.$$

\end{lemma}

\begin{proof}
The equation $E_0 = B_0^\beta =0$ holds true if and only of $B_0=0$. Additionally, because $B_0$ is a factor of all the $E_i$'s, $B_0=0$ implies that all the $E_i$'s are equal to $0$.

\end{proof}

To compute the multiplicity of the component $\Sigma_{(n)}$, we localize at a generic point $p$ of $V$. Note that $p$ is the prime ideal corresponding to $\Sigma_{(n)}^{red}$.

\begin{corollary}
\label{cor_multcount}
The local ring $(R_{(n)})_p$ is isomorphic to $$\mathbb{C}[b_0, \dots, b_{n-1}, x_1]_p/(B_0^\beta, B_0^{\beta -1}, B_0^{\beta -2}, \dots B_0^{\beta -n+1})_p,$$ and it has dimension $\beta -n +1$.

We conclude that the multiplicity of $\Sigma_{(n)}$ is $\beta -n +1$.
\end{corollary}

\begin{proof}
By Lemma \ref{divisiblelem}, we can factor $E_i$ into $E_i = B_0^{\beta -i}F_i$ where $F_i$ is a factor not divisible by $B_0$. Therefore, we can rewrite the local ring $(R_{(n)})_p$ as $$\mathbb{C}[b_0, \dots, b_{n-1}, x_1]_p/(B_0^\beta, \dots, B_0^{\beta -i}, \dots, B_0^{\beta-n+1})_p.$$ Recall that $B_0= b_0 +b_1x_1+b_2x_1^2 + \dots + b_{n-1}x_1^{n-1}$.
So the local ring $(R_{(n)})_p$ has basis $1, B_0, \dots, B_0^{\beta -n}$, and the dimension of the local ring $(R_{(n)})_p$ is $\beta -n +1$.
\end{proof}

\subsection{The multiplicity of a general stratum $\Sigma_{m_1, \dots, m_s}$}
\
\\

Now we consider all the strata $\Sigma_{m_1, \dots, m_s}$ whose closures are the irreducible components of $\Hilb^n(\{y^\beta =0\})$, so $\sum_i m_i =n$ and $m_s \leq \beta$. Recall that every irreducible component intersects $U_{(n)}$ and the intersection is open and dense in  $\Sigma_{m_1, \dots, m_s}$, so we study the multiplicities of the points of each irreducible components in the chart $U_{(n)}$.

The strategy of this section is to show that the coordinate ring of the stratum $\Sigma_{m_1, \dots, m_s}$ in $U_{(n)}$ is isomorphic as a $\mathbb{C}$-vector space to a tensor product of the coordinate rings of the strata $\Sigma_{(m_i)}$ of $\Hilb^{m_i}(\{y^\beta =0\})$, where $1 \leq i \leq n$, which we computed in the last section. Therefore the multiplicity of $\Sigma_{m_1, \dots, m_s}$ is the product of the multiplicities of $\Sigma_{m_i}$.
\\

We define the constructions to show this isomorphism. 




 We first recall our coordinate systems. Denote an ideal in $\Sigma_{m_1, \dots, m_s} \cap U_{(n)}$ as $((x-x_1)^{m_1}\dots(x-x_s)^{m_s}, y-b(x))$ for some $x_1, \dots, x_s \in \mathbb{C}$ satisfying $x_i \neq x_j$ for all $1 \leq i,j \leq s$.

 Denote an ideal in $\Sigma_{(m_i)} \cap U_{(m_i)} \subset \Hilb^{m_i}(\{y^\beta =0\})$, where $(m_i)$ is the one-part partition of the number $m_i$, as $((x-x_i)^{m_i}, y- b^i(x))$ where $b^i(x) = b^i_0 + b^i_1x+ \dots + b^i_{m_i-1}x^{m_i-1}$, and the coordinates are $b^i_0, \dots, b^i_{m_i-1}, x_i$.

\begin{definition}

(1)Work in $U_{(n)} \subset \Hilb^n(\{y^\beta =0\})$. 

Define $r(x) \coloneqq b^\beta(x) $ mod $a(x)$, the remainder of polynomial long division. 

Denote $r(x) \coloneqq r_0 + r_1x + \dots r^{n-1}x$. 
\\

(2)Work in $U_{(m_i)} \subset \Hilb^{m_i}(\{y^\beta =0\})$. 

Define $r^i(x) \coloneqq (b^i(x))^\beta$ mod $(x-x_i)^{m_i}$. 

Denote $r^i(x) = r^i_0 + r^i_1x + \dots r^i_{m_i-1}x^{m_i-1}$. 
\\



\end{definition}

We remark that $b_0, \dots, b_{n-1}$ and $b^i_0, \dots, b^i_{m_i-1}$ are formal variables as generators of coordinate rings of the corresponding strata. But $r_0, \dots, r_{n-1}$ are polynomials of variables $b_0, \dots, b_{n-1}, x_1, \dots, x_s$, and $r^i_0, \dots, r^i_{m_i-1}$ are polynomials of variables $b^i_0, \dots, b^i_{m_i-1}, x_i$.

\begin{lemma}
Immediately following the definitions, the coordinate ring $R_\mu$ of the scheme $\Sigma_{m_1, \dots, m_s}$ in $U_{(n)}$ of $\Hilb^n(\{y^\beta =0\})$ is given by $$R_\mu = \mathbb{C}[b_0, \dots, b_{n-1}, x_1, \dots, x_s]/(r_0, \dots, r_{n-1}).$$ 
\end{lemma}

Recall from the last subsection that the coordinate ring $R_{(m_i)}$ of the scheme $\Sigma_{(m_i)}$ in $U_{(m_i)}$ of $\Hilb^{m_i}(\{y^\beta =0 \})$ is given by $$R_{(m_i)} = \mathbb{C}[b^i_0, \dots, b^i_{m_i-1}, x_i]/(r^i_0, \dots, r^i_{m_i-1}).$$

Let $p$ be a generic point of $R_\mu$, such that $x_i \neq x_j$ for all $i \neq j$. We localize $R_\mu$ at $p$ and compute the $\mathbb{C}$-dimension of the ring $(R_\mu)_p$ as the multiplicity of the stratum $\Sigma_{m_1, \dots, m_s}$.

We now state the proposition below, that allows us to compute the dimension of $(R_\mu)_p$ by the dimension of $(R_{(m_i)})_p$. Recall that the dimension of $(R_{(m_i)})_p$ is $\beta - m_i +1$ as computed in Corollary \ref{cor_multcount}.

\begin{proposition}
\label{mult}
The following two local algebras are isomorphic: $(R_\mu)_p \cong \bigotimes_{i=1, \dots, s} (R_{(m_i)})_p$. Specifically,

$$\frac{\mathbb{C}[x_0, \dots, x_s, b_0, \dots, b_{n-1}]_p}{(r_0, \dots, r_{n-1})_p} \cong \bigotimes_{i=1, \dots, s} \frac{\mathbb{C}[x_i, b^i_0, \dots, b^i_{m_i-1}]_p}{(r^i_0, \dots, r^i_{m_i-1})_p}.$$
\end{proposition}

\begin{definition}
We define the ring homomorphism $\phi$ as follows. $$\phi: \bigotimes_{i=1, \dots, s} \frac{\mathbb{C}[x_i, b^i_0, \dots, b^i_{m_i-1}]_p}{(r^i_0, \dots, r^i_{m_i-1})_p} \rightarrow \frac{\mathbb{C}[x_0, \dots, x_s, b_0, \dots, b_{n-1}]_p}{(r_0, \dots, r_{n-1})_p}.$$

Define $\phi$ to be identity on the variables $x_i$, $\phi(x_i) = x_i$. And define $\phi(b_j^i)$ to be the coefficient of the term $x^j$ in the polynomial long division $b(x)$ mod $(x-x_i)^{m_i}$.

We define the image of $\phi$ on the generators, and extend the map $\phi$ to the entire ring by declaring that $\phi$ is a ring isomorphism. i.e. for any element $f$ in the ring $\bigotimes_{i=1, \dots, s} (R_{(m_i)})_p$, define $$\phi(f(x_1, \dots, x_s, b_0^1, \dots, b^i_j, \dots, b^s_{m_s-1})) = f(x_1, \dots, x_s, \phi(b_0^1), \dots, \phi(b^i_j), \dots, \phi(b^s_{m_s-1})
))$$
\end{definition}

\begin{remark}
We can add a variable $x$ and extend $\phi$ to a ring homomorphism from $\bigotimes_{i=1, \dots, s} (R_{(m_i)})_p[x]$ to $(R_\mu)_p[x]$ by sending $\phi(x)=x$. This homomorphism satisfies that $\phi(b^i(x)) = b(x)$ mod $(x-x_i)^{m_i}$ by construction. Due to the construction that $\phi$ is a ring homomorphism, we also have $\phi((b^i(x)^\beta) = \phi((b^i(x))^\beta =b^\beta(x)$ mod $(x-x_i)^{m_i}$.
\end{remark}

\begin{lemma}
\label{lemr}
We have that $\phi(r^i(x)) = r(x)$ mod $(x-x_i)^{m_i}$.
\end{lemma}

\begin{proof}
Because $(x-x_i)^{m_i}$ is a factor of $a(x)$, we have that 

$\phi(r^i(x)) = r^i(x, x_1, \dots, x_s, \phi(b^i_j)) = (b^i(x,x_1, \dots, x_s, \phi(b^i_j)))^\beta $ mod $(x-x_i)^{m_i} = (b(x))^\beta $ mod $(x-x_i)^{m_i} = r(x)$ mod $(x-x_i)^{m_i}$.
\end{proof}

\begin{proof}[Proof of Proposition \ref{mult}]

We first want to show that the rings 
$\bigotimes_{i=1, \dots, s} \limits \mathbb{C}[x_i, b^i_0, \dots, b^i_{m_i-1}]_p$ and $\mathbb{C}[x_0, \dots, x_s, b_0, \dots, b_{n-1}]_p$ are isomorphic by proving that $\phi$ is a linear change of variables between the ring generators. 

We write $\phi$ as a change-of-basis matrix with polynomial entries of variables $x_i$ that changes basis from the ring generators $b_0^{1}, \dots, b_j^{i}, \dots, b_{m_s-1}^{s}$ to the ring generators $b_0, \dots, b_{n-1}$. By construction, each $b^i_j$ is sent to a linear combination of $b_0, \dots, b_{n-1}$ with polynomial coefficients of variables $x_i$.

Now we show the other direction that each $b_j$ can be written as a linear combination of $b^i_j$ with coefficients being rational functions in $x_1, \dots, x_s$, such that the denominator of each rational function is a product of $x_i-x_l$ where $i \neq l$.

First, we want to show that the determinant of the matrix $\phi$ is a product of $x_i-x_l$ where $i \neq l$. Consider the $x_i$'s not as variables, but fixed numbers in $\mathbb{C}$, and assume all $x_i \neq x_l$ for all $i,l \in \{1, \dots, s\}$. Similarly, consider $b_0, \dots, b_{n-1}$ not as variables but fixed numbers in $\mathbb{C}$. So $b(x), b^i(x), r(x), r^{i}(x)$, $(x-x_1)^{m_1}\dotsc(x-x_s)^{m_s}$ and $(x-x_i)^{m_i}$ are all polynomials with complex coefficients in one variable $x$.

We have that $\frac{\mathbb{C}[x]}{((x-x_1)^{m_1}\dotsc(x-x_s)^{m_s})} \cong \bigotimes_i \frac{\mathbb{C}[x]}{((x-x_i)^{m_i})}$ by the Chinese remainder theorem. We can vary the values of $b_0, \dots, b_{n-1}$ such that $b(x)$ gives an arbitrary element of $\mathbb{C}[x]$. By the Chinese remainder theorem, $\phi$ is an isomorphism for all $x_i \neq x_l$.

Now we let $x_i$ be variables and consider $\det(\phi)$ as a non-constant function in $\mathbb{C}[x_1, \dots, x_s]$. For any set of values of $b_0, \dots, b_{n-1}$ and $x_1, \dots, x_s$ such that $i \neq l$, we have that $\phi$ is an invertible linear map and $\det(\phi) \neq 0$ for any $x_i \neq x_l$. Then $\det(\phi)$ does not vanish on the quasi-affine space defined by equations $\{x_i \neq x_l |\ \forall i \neq l\}$. So $\det(\phi)$ can only have factors that are $x_i - x_l$ for some $i \neq l$.

When we assume that $x_i \neq x_l$, the determinant $\det(\phi)$ never vanishes, and $\phi$ is invertible. By construction, $\phi^{-1}$ maps each $b_j$ to a linear combination of $b^i_j$ with coefficients being rational functions in $x_1, \dots, x_s$, such that the denominator of each rational function is $\det(\phi)$, a product of $x_i-x_l$ where $i \neq l$.

In the localized ring, $x_i \neq x_l$, so we have the ring isomorphism $$\bigotimes_{i=1, \dots, s} \mathbb{C}[x_i, b^i_0, \dots, b^i_{m_i-1}]_p \cong \mathbb{C}[x_0, \dots, x_s, b_0, \dots, b_{n-1}]_p.$$

By Lemma \ref{lemr}, $\phi(r^i_j)$ is the coefficient of the term $x^j$ in the polynomial long division $r(x)$ mod $(x-x_i)^{m_i}$. So by the property of $\phi$ we just showed, $\phi$ linearly changes variables between the collection $r^i_j$ and the collection $r_j$. So the ideals in the corresponding polynomial rings satisfy $(r_0, \dots, r_{m_i-1})_p \cong (r^1_0, \dots, r^i_j, \dots, r^s_{m_s-1})_p$.

The isomorphism of the quotient rings follows from the isomorphisms of the polynomial rings and the ideals we're quotienting them out with.
\end{proof}

\begin{corollary}
The multiplicity at a generic point $p$ of $\Sigma_{m_1, \dots, m_s}$ is $\prod_i (\beta - m_i +1)$.
\end{corollary}

\begin{proof}
The multiplicity of $\Sigma_{m_1, \dots, m_s}$ is the $\mathbb{C}$-dimension of the algebra $\frac{\mathbb{C}[x_0, \dots, x_s, b_0, \dots, b_{n-1}]_p}{(r_0, \dots, r_{n-1})_p}$. By the isomorphism in the theorem above, its dimension is equal to the product of the dimensions of $\frac{\mathbb{C}[x_i, b^i_0, \dots, b^i_{m_i-1}]_p}{(r^{i}_0, \dots, r^{i}_{m_i-1})_p}$, which are $\beta - m_i +1 $ as we computed in the last section.
\end{proof}

\subsection{Generalization: multiplicity of the irreducible components of any curves $C$}

\begin{theorem}
Recall that the irreducible components of $\Hilb^n(C)$ are indexed by partitions $m_1^1, \dots, m_i^j, \dots, m_{s_r}^r$ and each $m_i^j \leq \beta_j$ for all $i$ and $j$. The multiplicity of the component indexed by $m_1^1, \dots, m_i^j, \dots, m_{s_r}^r$ is $\Pi_{i,j}(\beta -m_i^j +1)$.

\end{theorem}

\begin{proof}
Locally, the component $\Sigma_{m_1^1, \dots, m_{s_1}^1, \dots, m_1^r, \dots, m_{s_r}^r}$ have the same multiplicity as the product of the components: $\Sigma_{m_1^1, \dots, m_{s_1}^1} \times \dots \times \Sigma_{m_1^r, \dots, m_{s_r}^r}$, where each $\Sigma_{m_1^j, \dots, m_{s_r}^j}$ is a stratum of $\Hilb^n(\{y^{\beta_j} = 0\})$. Therefore the multiplicity of the component indexed by $m_1^1, \dots, m_{s_1}^1, \dots, m_1^r, \dots, m_{s_r}^r$ is $\Pi_{i,j} (\beta_j - m_i^j+1)$.
\end{proof}

\end{document}